\documentclass[reqno,10pt]{amsart}

\usepackage{times}
\usepackage[T1]{fontenc}
\usepackage{mathrsfs}
\usepackage{hyperref}

\usepackage{latexsym}
\usepackage[dvips]{graphics}
\usepackage{epsfig}
\usepackage{amsmath,amsfonts,amsthm,amssymb,amscd}
\input amssym.def
\input amssym.tex
\usepackage{color}
\usepackage{graphicx}
\usepackage{verbatim}

\newtheorem{thm}{Theorem}
\newtheorem{lem}[thm]{Lemma}
\newtheorem{remark}[thm]{Remark}

\newtheorem{defn}[thm]{Definition}
\newtheorem{prop}[thm]{Proposition}

\numberwithin{thm}{section}
\numberwithin{equation}{section}

\newcommand\be{\begin{equation}}
\newcommand\ee{\end{equation}}
\newcommand\bea{\begin{eqnarray}}
\newcommand\eea{\end{eqnarray}}
\newcommand\bi{\begin{itemize}}
\newcommand\ei{\end{itemize}}
\newcommand\ben{\begin{enumerate}}
\newcommand\een{\end{enumerate}}

\newcommand{\kkot}[1]{ \frac{\sin \pi {#1} }{\pi {#1} } }
\newcommand{\foh}{\frac12}
\newcommand{\intii}{\int_{-\infty}^\infty}

\renewcommand{\Re}{\operatorname{\mathfrak{Re}}}

\newcommand{\sgn}[1]{\mathop{\mathrm{sgn}}(#1)}

\newcommand{\Z}{\mathbb{Z}}
\newcommand{\R}{\mathbb{R}}

\newcommand{\C}{\mathbb{C}}

\newcommand{\ord}{\mathop{\mathrm{ord}}}

\newcommand{\Avg}{\mathop{\mathrm{Avg}}}

\newcommand{\SL}{\mathrm{SL}}
\newcommand{\GL}{\mathrm{GL}}

\newcommand{\eps}{\epsilon}

\renewcommand{\sgn}{\mathrm{sgn}}

\newcommand{\hfrak}{\mathfrak{h}}

\begin{document}

\title{Low-lying zeros of Maass form $L$-functions}
\author{Levent Alpoge}
\email{alpoge@college.harvard.edu}
\address{Department of Mathematics, Harvard University, Cambridge, MA 02138 }

\author{Steven J. Miller}
\email{sjm1@williams.edu, Steven.Miller.MC.96@aya.yale.edu}
\address{Department of Mathematics, Williams College, Williamstown, MA 01267}

\subjclass[2010]{11M26 (primary), 11M41, 15A52 (secondary).}

\keywords{Low lying zeros, one level density, Maass form, Kuznetsov trace formula.}

\thanks{The first-named author was partially supported by NSF grant DMS0850577 and the second-named author by NSF grants DMS0970067 and DMS1265673. It is a pleasure to thank Andrew Knightly, Peter Sarnak, and our colleagues from the Williams College 2011 and 2012 SMALL REU programs for many helpful conversations.}

\begin{abstract}
The Katz-Sarnak density conjecture states that the scaling limits of the distributions of zeros of families of automorphic $L$-functions agree
with the scaling limits of eigenvalue distributions of classical subgroups of the unitary groups $U(N)$. This conjecture is often tested by way of
computing particular statistics, such as the one-level density, which evaluates a test function with compactly supported Fourier transform at
normalized zeros near the central point. Iwaniec, Luo, and Sarnak studied the one-level densities of cuspidal newforms of weight $k$ and level $N$. They showed in the limit as $kN \to\infty$ that these families have one-level densities agreeing with orthogonal type for test functions with Fourier transform supported in $(-2,2)$. Exceeding $(-1,1)$ is important as the three orthogonal groups are indistinguishable for support up to $(-1,1)$ but are distinguishable for any larger support. We study the other family of ${\rm GL}_2$ automorphic forms over $\mathbb{Q}$: Maass forms. To facilitate the analysis, we use smooth weight functions in the Kuznetsov formula which, among other restrictions, vanish to order $2M$ at the origin. For test functions with Fourier transform supported inside $\left(-2 + \frac{2}{2M+1}, 2 - \frac{2}{2M+1}\right)$, we unconditionally prove the one-level density of the low-lying zeros of level 1 Maass forms, as the eigenvalues tend to infinity, agrees only with that of the scaling limit of orthogonal matrices.
\end{abstract}


\maketitle

\tableofcontents

\section{Introduction}

The zeros of $L$-functions, especially those near the central point, encode important arithmetic information. Understanding their distribution has numerous applications, ranging from bounds on the size of the class numbers of imaginary quadratic fields \cite{CI,Go,GZ} to the size of the Mordell-Weil groups of elliptic curves \cite{BSD1,BSD2}. We concentrate on the one-level density, which allows us to deduce many results about these low-lying zeros.

\begin{defn} Let $L(s,f)$ be an $L$-function with zeros in the critical strip $\rho_f = 1/2 + i \gamma_f$ (note $\gamma_f \in \mathbb{R}$ if and only if the Grand Riemann Hypothesis holds for $f$), and let $\phi$ be an even Schwartz function whose Fourier transform has compact support. The \textbf{one-level density} is \begin{equation}\label{one level density}D_1(f;\phi,R)\ := \ \sum_{\rho_f} \phi\left(\frac{\log{R}}{2\pi}\gamma_f\right),\end{equation} where $R$ is a scaling parameter. Given a family $\mathcal{F}$ of $L$-functions and a weight function $w$ of rapid decay, we define the \textbf{averaged one-level density} of the family by \begin{equation} \mathcal{D}_1(\mathcal{F};\phi) \ := \  \lim_{R\to\infty} \frac1{W(\mathcal{F},R)} \sum_{f\in\mathcal{F}} w(C_f/R) D_1(f;\phi,R), \end{equation} with \begin{equation} W(\mathcal{F},R) \ := \ \sum_{f\in\mathcal{F}} w(C_f/R), \end{equation} and $C_f$ some normalization constant associated to the form $f$ (typically it is related to the analytic conductor $c_f$, e.g. $C_f = c_f$ or $c_f^{1/2}$, etc.).
\end{defn}

The Katz-Sarnak density conjecture \cite{KaSa1, KaSa2} states that the scaling limits of eigenvalues of classical compact groups near 1 correctly model the behavior of these zeros in families of $L$-functions as the conductors tend to infinity. Specifically, if the symmetry group is $\mathcal{G}$, then for an appropriate choice of the normalization $R$ we expect \begin{equation} \mathcal{D}_1(\mathcal{F};\phi) \ = \ \int_{-\infty}^\infty \phi(x) W_{1,\mathcal{G}}(x) dx  \ = \ \int_{-\infty}^\infty \widehat{\phi}(t) \widehat{W_{1,\mathcal{G}}}(t) dt, \end{equation} where $K(y) = \kkot{y}$, $ K_\epsilon(x,y) = K(x-y) + \epsilon K(x+y)$ for $\epsilon = 0, \pm 1$, and
\begin{eqnarray}\label{eqdensitykernels}
W_{1,\mathrm{SO(even)}}(x) &\ =\ &  K_1(x,x) \nonumber\\ W_{1,\mathrm{SO(odd)}}(x) & = & K_{-1}(x,x)  + \delta_0(x) \nonumber\\ W_{1,\mathrm{O}}(x) & =
& \foh W_{1,\mathrm{SO(even)}}(x) + \foh W_{1,\mathrm{SO(odd)}}(x) \nonumber\\ W_{1,\mathrm{U}}(x) & =
& K_0(x,x) \nonumber\\ W_{1,\mathrm{Sp}}(x) &=&  K_{-1}(x,x).
\end{eqnarray}
Note the Fourier transforms of the densities of the three orthogonal groups all equal $\delta_0(y) + 1/2$ in the interval $(-1,1)$ but are mutually distinguishable for larger support (and are distinguishable from the unitary and symplectic cases for any support). Thus if the underlying symmetry type is believed to be orthogonal then it is necessary to obtain results for test functions $\phi$ with ${\rm supp}(\widehat{\phi})$ exceeding $(-1,1)$ in order to have a unique agreement.

The one-level density has been computed for many families for suitably restricted test functions, and has always agreed with a random matrix ensemble. Simple families of $L$-functions include Dirichlet $L$-functions, elliptic curves, cuspidal newforms, number field $L$-functions, and symmetric powers of ${\rm GL}_2$ automorphic representations \cite{DM1,FiMi,FI,Gao,GK,Gu,HM,HR,ILS,IMT,KaSa1,KaSa2,Mil,MilPe,OS1,OS2,RR,Ro,Rub1,Rub2,ShTe,Ya,Yo}. Due\~nez and Miller \cite{DM1,DM2} handled some compound families, and recently Shin and Templier \cite{ShTe} determined the symmetry type of many families of automorphic forms on ${\rm GL}_n$ over $\mathbb{Q}$. The goal of this paper is to provide additional evidence for these conjectures for the family of level 1 Maass forms for as large support of the test function as possible.


\subsection{Background and Notation}

By $A\ll B$ we mean that $|A|\le c|B|$ for some positive constant $c$, and by $A\asymp B$ we mean that $A\ll B$ and $B\ll A$. We set \begin{equation}e(x)\ :=\ \exp(2\pi i x)\end{equation} and use the following convention for the Fourier transform: \begin{equation}\widehat{f}(\xi)\ :=\ \intii f(x)e(-x\xi) dx.\end{equation}

We quickly review some properties of Maass forms; see \cite{Iw1,IK,KL,Liu,LiuYe1,LiuYe2} for a detailed exposition and a derivation of the Kuznetsov trace formula, which will be a key ingredient in our analysis below.

Let $u$ be a cuspidal (Hecke-Maass-Fricke) eigenform on $\SL_2(\Z)$ with Laplace eigenvalue $\lambda_u=:\frac{1}{4} + t_u^2, t_u\in \C$. By work of Selberg we may take $t_u\geq 0$. We may write the Fourier expansion of $u$ as \begin{equation}u(z)\ =\ y^{1/2}\sum_{n\neq 0} a_n(u)K_{s-1/2}(2\pi |n| y)e(ny).\end{equation} Let \begin{equation}\lambda_n(u)\ :=\ \frac{a_n(u)}{\cosh(t)^{1/2}}.\end{equation} Changing $u$ by a non-zero constant if necessary, by the relevant Hecke theory on this space without loss of generality we may take $\lambda_1 = 1$. This normalization is convenient in applying the Kuznetsov trace formula to convert sums over the Fourier coefficients of $u$ to weighted sums over prime powers.

The $L$-function associated to $u$ is \begin{equation}L(s,u) \ :=\ \sum_{n\geq 1} \lambda_n n^{-s}.\end{equation} By results from Rankin-Selberg theory the $L$-function is absolutely convergent in the right half-plane ${\Re}(s) > 1$ (one could also use the work of Kim and Sarnak \cite{K,KSa} to obtain absolutele convergent in the right half-plane ${\Re}(s) > 71/64$, which suffices for our purposes). These $L$-functions analytically continue to entire functions of the complex plane, satisfying the functional equation \begin{equation}\Lambda(s,u)\ =\ (-1)^\eps\Lambda(1-s,u),\end{equation} with \begin{equation}\Lambda(s,u) \ :=\ \pi^{-s}\Gamma\left(\frac{s + \eps + it}{2}\right)\Gamma\left(\frac{s+\eps-it}{2}\right)L(s,u).\end{equation} Factoring \begin{equation}1 - \lambda_pX + X^2\ =:\ (1-\alpha_pX)(1-\beta_pX)\end{equation} at each prime (the $\alpha_p,\beta_p$ are the Satake parameters at $p$), we get an Euler product \begin{equation}L(s,u)\ =\ \prod_p (1-\alpha_pp^{-s})^{-1}(1-\beta_pp^{-s})^{-1},\end{equation} which again converges for ${\Re}(s)$ sufficiently large.

We let $\mathcal{M}_1$ denote an orthonormal basis of Maass eigenforms, which we fix for the remainder of the paper. In what follows $\Avg(A;B)$ will denote the average value of $A$ over our orthonormal basis of level 1 Maass forms weighted by $B$. That is to say, \begin{equation} \Avg(A;B)\ :=\ \frac{\sum_{u\in \mathcal{M}_1} A(u)B(u)}{\sum_{u\in \mathcal{M}_1} B(u)}.\end{equation}

\subsection{Main result}

Before stating our main result we first describe the weight function used in the one-level density for the family of level 1 Maass forms. The weight function we consider is not as general as other ones investigated (see the arguments for other families of Maass forms in \cite{AAILMZ}), but leads to a significantly simpler analysis and much greater support. In this sense our work is similar to analyses in other problems where the weight function is chosen to facilitate the application of a summation formula (for example, the use of harmonic weights for the Petersson formula). As previous work on Maass forms could not deal with test functions whose Fourier transforms are supported outside $(-1, 1)$, these calculations were insufficient to determine the underlying symmetry. As extending this support is the primary motivation for this work, we thus chose a weight function which is ideally suited for using the Kuznetsov trace formula.

As we will see below, some type of weighting is necessary in order to restrict to conductors of comparable size. While our choice does not include, say, the characteristic function of $[T, 2T]$, we are able to localize for the most part to conductors near $T$, with polynomial decay before $T$ and exponential decay beyond. By choosing such weight functions, we are able to unconditionally obtain support in $(-2, 2)$. Note this equals the best unconditional results for any family of $L$-functions, that of Dirichlet $L$-functions (support this large is known for cuspidal newforms, but only by assuming GRH for Dirichlet $L$-functions to expand the Kloosterman sums).

Let $h\in C^\infty\left(\R\right)$ be an even smooth function with an even smooth square-root of Paley-Wiener class such that $\hat{h}\in C^\infty\left(\left(-1/4, 1/4\right)\right)$ and $h$ has a zero of order at least $2M\geq 8$ at $0$. In fact, the higher the order of the zero of $h$ at $0$, the better the support we are able to obtain: this will be made precise below.

By the ideas that go into the proof of the Paley-Wiener theorem, since $\widehat{h}$ is compactly supported we have that $h$ extends to an entire holomorphic function, with the estimate \begin{equation}h(x+iy)\ \ll\ \exp\left(\frac{\pi |y|}{2}\right).\end{equation} Note also that, by exhibiting $h$ as the square of a real-valued even smooth function on the real line (that also extends to an entire holomorphic function by Paley-Wiener), by the Schwarz reflection principle we have that $h$ takes non-negative real values along the imaginary axis as well.\\

\textbf{Throughout this paper $T$ will be a large positive odd integer tending to infinity.}  \\ \

Let \begin{equation}\label{weight function}h_T(r)\ :=\ \frac{\left(\frac{r}{T}\right)h\left(\frac{ir}{T}\right)}{\sinh\left(\frac{\pi r}{T}\right)}.\end{equation} For $r\in \R$ we have \begin{equation}h_T(r)\ \ll\ \exp\left(-\frac{\pi |r|}{4T}\right).\end{equation} Further, $h_T$ extends to an entire meromorphic function, with poles exactly at the non-zero integral multiples of $iT$. Figure \ref{weight function plot} shows a plot of $h_{101}(r)$ on $[0,1000]$ for one choice of $h$. The point is that $h_T(r)$ is order $1$ for $r$ on the order of $T$, decays exponentially at infinity, and decays polynomially at zero (like e.g.\ the Maxwell-Boltzmann distribution, and many other well-known distributions).

\begin{center}
\begin{figure}[!h]
\begin{center}
 \scalebox{1}{\includegraphics{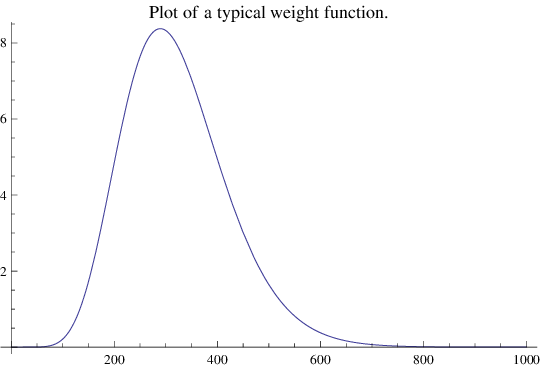}} \end{center}
	  \caption{A plot of $h_{101}$. Here $b(t) = \exp(-1/(1/100 - t^2))$ if $|t| \le 1/10$ and 0 otherwise, $h(\xi) = \xi^8 \widehat{b}(\xi)$ (where $\widehat{b}$ is the Fourier transform of $b$), and $h_T(r) = (r/T) h(ir/T) / \sinh(\pi r / T)$. \label{weight function plot}}
\end{figure}
\end{center}


In our one-level calculations we take our test function $\phi$ to be an even Schwartz function such that ${\rm supp}(\widehat{\phi}) \subset \left(-\eta,\eta\right)$ for some $\eta > 0$. The goal of course is to prove results for the largest $\eta$ possible. We suppress any dependence of constants on $h$ or $\eta$ or $\phi$ as these are fixed, but not on $T$ as that tends to infinity.

%



In computing the one-level density for the family $\mathcal{M}_1$, we have some freedom in the choice of weight function. We choose to weight $u$ by $h_T(t_u)/||u||^2$, where $t_u^2 + 1/4$ is the Laplace eigenvalue of $u$, and $||u|| = ||u||_{L^2(\SL_2(\Z)\backslash\hfrak)}$ is the $L^2$ norm of $u$. We may write the averaged one-level density as (we will see that $R\asymp T^2$ is forced) \begin{eqnarray} \mathcal{D}_1(\mathcal{M}_1; \phi) & \ = \ & \lim_{T\to\infty \atop T\ {\rm odd}} \frac1{\sum_{u \in \mathcal{M}_1} h_T(t_u)/||u||^2} \sum_{u\in\mathcal{M}_1} D_1(u;\phi,T^2) \frac{h_T(t_u)}{||u||^2}\nonumber\\ &=& \lim_{T\to\infty \atop T\ {\rm odd}} \Avg\left(D_1(u;\phi,T^2); \frac{h_T(t_u)}{||u||^2}\right).
\end{eqnarray}

Based on results from \cite{AAILMZ} and \cite{ShTe}, which determined the one-level density for support contained in $(-1, 1)$, we believe the following conjecture.

\ \\
\noindent \textbf{Conjecture:} \emph{Let $h_T$ be as defined in \eqref{weight function} and $\phi$ an even Schwartz function with $\widehat{\phi}$ of compact support. Then \be \mathcal{D}_1(\mathcal{M}_1; \phi) \ = \  \int_{-\infty}^\infty \phi(t) W_{1, \mathrm{O}}(t)dt, \ee with $W_{1, \mathrm{O}}(t) = 1 + \frac12 \delta_0$. In other words, the symmetry group associated to the family of level 1 cuspidal Maass forms is orthogonal.} \ \\


Unfortunately, the previous one-level calculations are insufficient to distinguish which of the three orthogonal candidates is the correct corresponding symmetry type, as they all agree in the regime calculated. There are two solutions to this issue. The first is to compute the two-level density, which is able to distinguish the three candidates for arbitrarily small support (see \cite{Mil}). The second is to compute the one-level density in a range exceeding $(-1,1)$, which we do here.

Before stating the main result, it is worth mentioning that padding the weight function with more zeros at $0$ allows us to increase the support to $(-2 + \eps, 2 - \eps)$ for any $\epsilon > 0$; note that we do not assume GRH. This equals the best support obtainable either unconditionally or under just GRH for any family of $L$-functions (such as Dirichlet $L$-functions \cite{FiMi,Gao,HR,OS1,OS2} and cuspidal newforms not split by sign \cite{ILS}), and thus provides strong evidence for the Katz-Sarnak density conjecture for this family. It is also worth noting that the methods employed in the proof fail at almost every stage if we have support outside $(-2,2)$, so this is indeed a natural barrier. Having said this, we may now state the main theorem.

\begin{thm}\label{main theorem}
Let $T > 1$ be an odd integer and $\phi$ an even Schwartz function with ${\rm supp}(\widehat{\phi}) \subset (-\eta, \eta)$. Let $h\in C^\infty\left(\R\right)$ be an even smooth function with an even smooth square-root of Paley-Wiener class such that $\widehat{h}\in C^\infty\left(\left(-1/4, 1/4\right)\right)$ and $h$ has a zero of order at least $2M\geq 8$ at $0$. Let $h_T$ be as defined in \eqref{weight function}. Then, for all $\eta < 2 - \frac{2}{2M+1}$, we have that \be \mathcal{D}_1(\mathcal{M}_1; \phi) \ = \  \int_{-\infty}^\infty \phi(t) W_{1, \mathrm{O}}(t)dt, \ee the density corresponding to the orthogonal group, $\mathrm{O}$. That is to say, the symmetry group associated to the family of level 1 cuspidal Maass forms is orthogonal.
\end{thm}


\subsection{Outline of proof}

We give a quick outline of the argument. We carefully follow the seminal work of Iwaniec-Luo-Sarnak \cite{ILS} in our preliminaries. Namely, we first write down the explicit formula to convert the relevant sums over zeros to sums over Hecke eigenvalues. We then average and apply the Kuznetsov trace formula to leave ourselves with calculating various integrals, which we then sum. To be slightly more specific, we reduce the difficulty to bounding an integral of shape
\begin{align}
\intii J_{2ir}(X)\frac{r h_T(r)}{\cosh(\pi r)} dr,
\end{align}
where these $J$ are Bessel functions, and $h_T$ is as in Theorem \ref{main theorem}. We break into cases: $X$ ``small'' and $X$ ``large''. For $X$ small, we move the line of integration from $\R$ down to $\R - iR$ and take $R\to +\infty$, converting the integral to a sum over residues. The difficulty then lies in bounding a sum of residues of shape
\begin{align}
T\sum_{k\geq 0} (-1)^k J_{2k+1}(X) \frac{P\left(\frac{2k+1}{2T}\right)}{\sin\left(\frac{2k+1}{2T}\pi\right)},
\end{align}
where $P$ is closely related to $h$.
To do this (after a few tricks), we apply an integral formula for these Bessel functions, switch summation and integration, apply Poisson summation, apply Fourier inversion, and then \emph{apply Poisson summation again}. The result is a sum of Fourier coefficients, to which we apply the stationary phase method one by one. This yields the bound for $X$ small.

To handle $X$ large, we use a precise asymptotic for the $J_{2ir}(X)$ term from Dunster \cite{Du} (as found in \cite{ST}). In fact, for $X$ large it is enough to simply use the oscillation of $J_{2ir}(X)$ to get cancelation. It is worth noting that the same considerations would also be enough for the case of $X$ small were the asymptotic expansion convergent.



\section{Calculating the averaged one-level density}


The starting point is to use the explicit formula to convert weighted averages of the Fourier coefficients to weighted sums over prime powers. The calculation is standard and easily modified from \cite{RS} (see also Lemma 2.8 of \cite{AAILMZ}).

\begin{lem}[Explicit formula] Let $h_T$ be as in Theorem \ref{main theorem}. Then
\begin{eqnarray}
\Avg\left(D_1(u;\phi,R); \frac{h_T(t_u)}{||u||^2}\right) & \ = \ & \frac{\phi(0)}{2} + \hat{\phi}(0)\frac{\Avg\left(\log(1 + t_u^2); \frac{h_T(t_u)}{||u||^2}\right)}{\log{R}} \nonumber\\ & & \ \ -\ \sum_p \frac{2\log{p}}{p^{1/2}\log{T}}\hat{\phi}\left(\frac{\log{p}}{2\log{T}}\right)\Avg\left(\lambda_p(u); \frac{h_T(t_u)}{||u||^2}\right) \nonumber\\ & & \ \ -\ \sum_p \frac{2\log{p}}{p\log{T}}\hat{\phi}\left(\frac{\log{p}}{\log{T}}\right)\Avg\left(\lambda_{p^2}(u);\frac{h_T(t_u)}{||u||^2}\right) \nonumber\\ & & \ \ + \ O\left(\frac{\log\log{T}}{\log{T}}\right).
\end{eqnarray}
\end{lem}

To prove Theorem \ref{main theorem}, it therefore suffices to show the following.

\begin{lem}\label{conductor, prime sum, and prime square sum} Let $h_T$ be as in Theorem \ref{main theorem}. Then as $T \to\infty$ through the odd integers we have
\begin{eqnarray} & & (1) \ \ \ \Avg\left(\log(1 + t_u^2); \frac{h_T(t_u)}{||u||^2}\right)\ =\ \log(T^2) + O\left(\log\log{T}\right) \nonumber\\
& & (2) \ \ \ \sum_p \frac{\log{p}}{p^{1/2}\log{T}}\hat{\phi}\left(\frac{\log{p}}{2\log{T}}\right)\Avg\left(\lambda_p(u); \frac{h_T(t_u)}{||u||^2}\right)\to 0 \ \ \ \ \ \ \ \ \ \ \ \ \ \ \ \ \ \ \ \ \ \ \ \ \ \ \ \ \ \   \nonumber\\
& & (3) \ \ \ \sum_p \frac{\log{p}}{p\log{T}}\hat{\phi}\left(\frac{\log{p}}{\log{T}}\right)\Avg\left(\lambda_{p^2}(u);\frac{h_T(t_u)}{||u||^2}\right)\to 0. \end{eqnarray}
\end{lem}

The first determines the correct scale to normalize the zeros, $R\asymp T^2$ (see \cite{Mil} for comments on normalizing each form's zeros by a local factor and not a global factor such as $T^2$ here; briefly if only the one-level density is being studied then either is fine). The third is far easier than the second. Each will be handled via the Kuznetsov trace formula (see for example \cite{IK,KL,LiuYe2}), which we now state.

\begin{thm}[Kuznetsov trace formula]\label{thm:kuznetsovtraceformula}
Let $m,n\in \Z^+$. Let $H$ be an even holomorphic function on the strip $\{x+iy\,\vert\, |y| < \frac{1}{2} + \eps\}$ (for some $\eps > 0$) such that $H(z)\ll \frac{1}{1+y^2}$. Then
\begin{align}
\sum_{u\in \mathcal{M}_1} \frac{H(t_u)}{||u||^2} \lambda_m(u) \overline{\lambda_n(u)} &= \frac{\delta_{m,n}}{\pi^2}\intii rH(r)\tanh(\pi r) dr \nonumber\\&\quad- \frac{1}{\pi}\intii m^{ir}\sigma_{ir}(m)n^{-ir}\sigma_{-ir}(n)\frac{H(r)}{|\zeta(1 + 2ir)|^2}dr \nonumber\\&\quad+ \frac{2i}{\pi}\sum_{c\geq 1} \frac{S(m,n;c)}{c}\intii J_{2ir}\left(\frac{4\pi\sqrt{mn}}{c}\right)\frac{rH(r)}{\cosh(\pi r)}dr,
\end{align}
the sum taken over an orthonormal basis of Hecke-Maass-Fricke eigenforms on $\SL_2(\Z)$, with $S$ the usual Kloosterman sum, $\sigma$ the extended divisor function and $\delta_{m,n}$ Kronecker's delta.
\end{thm}

Observe that our weight function $h_T$ satisfies the hypotheses of the above theorem once $T > 1$, since the sine function has a simple zero at $0$.

Our first application of the Kuznetsov trace formula is to determine the total mass (i.e., the normalizing factor in our averaging).

\begin{lem}\label{denominator} Let $h_T$ be as in Theorem \ref{main theorem}. Then
\begin{align}
\sum_{u\in \mathcal{M}_1} \frac{h_T(t_u)}{||u||^2}\ \asymp\ T^2.
\end{align}
\end{lem}

\begin{proof}
We apply Theorem \ref{thm:kuznetsovtraceformula} to $h_T$, with $m=n=1$. We obtain
\begin{eqnarray}
\sum_{u\in \mathcal{M}_1} \frac{h_T(t_u)}{||u||^2} & \ = \ &\frac{1}{\pi^2}\intii rh_T(r)\tanh(\pi r) dr - \frac{1}{\pi}\intii \frac{h_T(r)}{|\zeta(1 + 2ir)|^2}dr \nonumber\\ & &\ \ +\ \frac{2i}{\pi}\sum_{c\geq 1} \frac{S(1,1;c)}{c}\intii J_{2ir}\left(\frac{4\pi}{c}\right)\frac{rh_T(r)}{\cosh(\pi r)}dr.
\end{eqnarray}

It is rather easy to see that the first term is $\asymp T^2$, since $h_T$ is non-negative and essentially supported on $r\asymp T$. Similarly, using $|\zeta(1+2ir)| \gg 1/\log(2+|r|)$ (see for example \cite{Liu}), the second term is readily seen to be \begin{equation}\label{Eisenstein term}\ll \ \frac{T\log{T}}{(mn)^{1/2}}.\end{equation} Applying the Weil bound, it certainly suffices to show that
\be
\intii J_{2ir}\left(\frac{4\pi}{c}\right)\frac{rh_T(r)}{\cosh(\pi r)}dr\ \ll\ c^{-1}.
\ee
But this follows from Proposition \ref{contour integration} and the bound
\begin{equation}
J_n(x)\ll \frac{\left(\frac{x}{2}\right)^n}{n!},
\end{equation} completing the proof.
\end{proof}

We can now prove the first part of the main lemma needed to prove Theorem \ref{main theorem}.

\begin{proof}[Proof of Lemma \ref{conductor, prime sum, and prime square sum}, part (1)]
We cut the sum above at $T\log{T}$ and below at $\frac{T}{\log{T}}$ and apply the previous lemma along with the fact that $||u||\asymp 1$ under our normalizations (see \cite{Smi}).
\end{proof}

We are thus left with the last two parts of Lemma \ref{conductor, prime sum, and prime square sum}.

\section{Handling the Bessel integrals}

In this section we analyze the Bessel terms. Crucial in our analysis is the fact that our weight function $h_T$ is holomorphic with nice properties; this allows us to shift contours and convert our integral to a sum over residues. The goal of the next few subsections is to prove the following two propositions, which handle $X$ small and large.

\begin{prop}\label{small bound} Let $h_T$ be as in \eqref{weight function}.
Suppose $X\leq T$. Then
\begin{align}
\intii J_{2ir}(X)\frac{rh_T(r)}{\cosh(\pi r)}dr\ \ll\ \frac{X}{T^2}.
\end{align}
\end{prop}

\begin{prop}\label{large bound} Let $h_T$ be as in \eqref{weight function} --- in particular, so that it has at least $M+1$ zeros at $0$.
Suppose $X\geq \frac{T}{8}$. Then
\begin{align}\label{large sum}
\intii J_{2ir}(X) \frac{r h_T(r)}{\cosh\left(\pi r\right)} dr\ \ll\ \frac{X^{M - \frac{1}{2}}}{T^{2M-2}} + \frac{T^2}{X^{5/2}}.
\end{align}
\end{prop}

\subsection{Calculating the Bessel integral}

We begin our analysis of the Bessel terms, which will eventually culminate in a proof of Proposition \ref{small bound}.

\begin{prop}\label{contour integration} Let $h_T$ be as in \eqref{weight function}. Then
\begin{eqnarray}
\intii J_{2ir}(X)\frac{rh_T(r)}{\cosh(\pi r)}dr &\ =\ & c_1\sum_{k\geq 0} (-1)^k J_{2k+1}(X) (2k+1) h_T\left(\left(k+\frac{1}{2}\right)i\right)\nonumber\\ & &\ \ +\ c_2 T \sum_{k\geq 1} (-1)^k J_{2kT}(X) k^2 h(k)\nonumber\\&=& c_1\sum_{k\geq 0} (-1)^k J_{2k+1}(X) (2k+1) h_T\left(\left(k+\frac{1}{2}\right)i\right) \nonumber\\ & & \ \ +\ O\left(Xe^{-c_3 T}\right),
\end{eqnarray}
where $c_1$, $c_2$, and $c_3$ are some constants independent of $X$ and $T$.
\end{prop}

\begin{proof}[Proof of Proposition \ref{contour integration}]
The idea here is to move the contour from $\R$ down to $\R-i\infty$, picking up poles at all the half-integers multiplied by $i$ (poles arising from the $\cosh(\pi r)$ in the denominator) and integer multiples of $iT$ (poles arising from the $\sinh\left(\frac{\pi r}{T}\right)$ hidden in $h_T$) that are passed. Indeed, the first sum is precisely the sum of the former residues, while the second is the sum of the latter. The final point is that $J_\alpha(z)$ decays extremely rapidly as $\Re{\alpha}\to \infty$, with $z$ fixed. One way to see this decay is to use the expansion \begin{equation}J_\alpha(2z)\ =\ \sum_{n\geq 0} \frac{(-1)^n z^{2n+\alpha}}{\Gamma(n+1)\Gamma(n+\alpha+1)},\end{equation} switch the sum and integral, and use Stirling's formula to do the relevant calculations, switching sums and integrals back at the end to consolidate the form into the above. The details will not be given here, as the bounds already given on $h_T$, as well as Stirling's bounds on $\Gamma$ (and the outline above), reduce this to a routine computation.

The claimed bound on the error term follows by trivially bounding by using (for $0\leq x\leq 1$, $n$ a positive integer)
\begin{equation}
\left|J_n(nx)\right|\leq \left(\frac{xe^{\sqrt{1-x^2}}}{1+\sqrt{1-x^2}}\right)^n,
\end{equation}
which can be found in \cite{AS}.
\end{proof}

\subsection{Averaging Bessel functions of integer order for small primes}

Iwaniec-Luo-Sarnak, in proving the Katz-Sarnak density conjecture for $\widehat{\phi}$ supported in $(-2,2)$ for holomorphic cusp forms of weight at most $K$, demonstrate a crucial lemma pertaining to averages of Bessel functions. In some sense our analogous work here moving this to the Kuznetsov setting requires only one more conceptual leap, which is to apply Poisson summation a second time to a resulting weighted exponential sum. The original argument can be found in Iwaniec's book (\cite{Iw2}), which we basically reproduce as a first step in handling the remaining sum from above.

\begin{remark}
We will use the fact that $J_{-n}(x) = (-1)^n J_n(x)$ several times in what follows. Moreover, we introduce the notation \begin{equation}\tilde{h}(x) \ :=\ xh(x),\end{equation} and similarly for iterated tildes.
\end{remark}

Thus (in this notation) to prove Proposition \ref{small bound} it suffices to show the following.

\begin{prop}\label{prop:smallboundSJ} Let $h_T$ be as in \eqref{weight function}.
Suppose $X\leq T$. Then
\begin{align}\label{actual small bound}
S_J(X)\ :=\ T\sum_{k\geq 0} (-1)^k J_{2k+1}(X) \frac{\tilde{\tilde{h}}\left(\frac{2k+1}{2T}\right)}{\sin\left(\frac{2k+1}{2T}\pi\right)}\ \ll\  \frac{X}{T^2}.
\end{align}
\end{prop}

\begin{proof}
Observe that $k\mapsto \sin\left(\frac{\pi k}{2}\right)$ is supported only on the odd integers, and maps $2k+1$ to $(-1)^k$. Hence, rewriting gives
\begin{align}\label{actual two}
S_J(X)\ =\ T\sum_{k\geq 0 \atop k\not\in 2T\Z} J_k(X) \tilde{\tilde{h}}\left(\frac{k}{2T}\right)\frac{\sin\left(\frac{\pi k}{2}\right)}{\sin\left(\frac{\pi k}{2T}\right)}.
\end{align}
As
\begin{align}
\frac{\sin\left(\frac{\pi k}{2}\right)}{\sin\left(\frac{\pi k}{2T}\right)}\ =\ \frac{e^{\frac{\pi i k}{2}} - e^{-\frac{\pi i k}{2}}}{e^{\frac{\pi i k}{2T}} - e^{\frac{\pi i k}{2T}}}\ =\ \sum_{\alpha = -\left(\frac{T-1}{2}\right)}^{\frac{T-1}{2}} e^{\frac{\pi i k \alpha}{T}}
\end{align}
when $k$ is not a multiple of $2T$, we find that
\begin{align}\label{actual three}
S_J(X)\ =\ T\sum_{|\alpha| < \frac{T}{2}} \sum_{k\geq 0 \atop k\not\in 2T\Z} e\left(\frac{k\alpha}{2T}\right) J_k(X) \tilde{\tilde{h}}\left(\frac{k}{2T}\right).
\end{align}
Observe that, since the sum over $\alpha$ is invariant under $\alpha\mapsto -\alpha$ (and it is non-zero only for $k$ odd!), we may extend the sum over $k$ to the entirety of $\Z$ at the cost of a factor of 2 and of replacing $h$ by \begin{equation}\label{definition of g}g(x)\ := \ \sgn(x)h(x).\end{equation} Note that $g$ is as differentiable as $h$ has zeros at $0$, less one. That is to say, $\widehat{g}$ decays like the reciprocal of a degree $\ord_{z=0} h(z) - 1$ polynomial at $\infty$. This will be crucial in what follows.

Next, we add back on the $2T\Z$ terms and obtain
\begin{eqnarray}\label{actual four}
\frac{1}{2}S_J(X) &\ =\ & T\sum_{|\alpha| < \frac{T}{2}} \sum_{k\in \Z} e\left(\frac{k\alpha}{2T}\right) J_k(X) \tilde{\tilde{g}}\left(\frac{k}{2T}\right) - T^2\sum_{k\in \Z} J_{2kT}(X) k^2 h(k)
\nonumber\\ &=& T\sum_{|\alpha| < \frac{T}{2}} \sum_{k\in \Z} e\left(\frac{k\alpha}{2T}\right) J_k(X) \tilde{\tilde{g}}\left(\frac{k}{2T}\right) + O\left(Xe^{-c_4T}\right)\nonumber\\ & =: & V_J(X) + O\left(Xe^{-c_4T}\right),
\end{eqnarray}
by the same argument as the last step of Proposition \ref{contour integration} (since the sign was immaterial).

Now we move to apply Poisson summation. Write $X=:2\pi Y$. We apply the integral formula (for $k\in \Z$) \begin{equation}J_k(2\pi x)\ =\ \int_{-\frac{1}{2}}^{\frac{1}{2}} e\left(kt - x\sin(2\pi t)\right) dt\end{equation} and interchange sum and integral (via rapid decay of $g$) to get that
\begin{align}\label{actual five}
V_J(X)\ =\ T\sum_{|\alpha| < \frac{T}{2}} \int_{-\frac{1}{2}}^{\frac{1}{2}} \left(\sum_{k\in \Z} e\left(\frac{k\alpha}{2T} + kt\right) \tilde{\tilde{g}}\left(\frac{k}{2T}\right)\right) e\left(-Y\sin(2\pi t)\right) dt.
\end{align}
By Poisson summation, \eqref{actual five} is just (interchanging sum and integral once more)
\begin{eqnarray}\label{actual six}
V_J(X) &\ =\ & T^2\sum_{|\alpha| < \frac{T}{2}} \sum_{k\in \Z}\int_{-\frac{1}{2}}^{\frac{1}{2}} \hat{g}''\left(2T(t-k) + \alpha\right) e\left(-Y\sin(2\pi t)\right) dt \nonumber\\&=& c_5T\sum_{|\alpha| < \frac{T}{2}} \int_{-\infty}^\infty \hat{g}''(t) e\left(Y\sin\left(\frac{\pi t}{T} + \frac{\pi \alpha}{T}\right)\right) dt\nonumber\\&=:&c_5W_g(X).
\end{eqnarray}

As \begin{equation}\sin\left(\frac{\pi t}{T} + \frac{\pi\alpha}{T}\right)\ =\ \sin\left(\frac{\pi\alpha}{T}\right) + \frac{\pi t}{T}\cos\left(\frac{\pi\alpha}{T}\right) - \frac{\pi^2 t^2}{T^2}\sin\left(\frac{\pi\alpha}{T}\right) - \frac{\pi^3 t^3}{T^3} \cos\left(\frac{\pi\alpha}{T}\right) + O\left(\frac{t^4}{T^4}\right),\end{equation} we see that (expanding $e(x) = 1 + 2\pi i x - 2\pi^2 x^2 + O(x^3)$ and using $Y\ll T$)
\begin{align}
W_g(X) &= c_6T\sum_{|\alpha| < \frac{T}{2}} e\left(Y\sin\left(\frac{\pi\alpha}{T}\right)\right) \int_{-\infty}^\infty \hat{g}''(t) e\left(\frac{\pi Y t}{T}\cos\left(\frac{\pi\alpha}{T}\right)\right) dt \nonumber\\&\hspace{.5in}+ c_7\frac{Y}{T}\sum_{|\alpha| < \frac{T}{2}} e\left(Y\sin\left(\frac{\pi\alpha}{T}\right)\right) \sin\left(\frac{\pi\alpha}{T}\right) \int_{-\infty}^\infty t^2\hat{g}''(t) e\left(\frac{\pi Y t}{T}\cos\left(\frac{\pi\alpha}{T}\right)\right) dt \nonumber\\&\hspace{.5in}+ c_8\frac{Y}{T^2}\sum_{|\alpha| < \frac{T}{2}} e\left(Y\sin\left(\frac{\pi\alpha}{T}\right)\right) \cos\left(\frac{\pi\alpha}{T}\right) \int_{-\infty}^\infty t^3\hat{g}''(t) e\left(\frac{\pi Y t}{T}\cos\left(\frac{\pi\alpha}{T}\right)\right) dt \nonumber\\&\hspace{.5in}+ c_9\frac{Y^2}{T^3}\sum_{|\alpha| < \frac{T}{2}} e\left(Y\sin\left(\frac{\pi\alpha}{T}\right)\right) \sin^2\left(\frac{\pi\alpha}{T}\right) \int_{-\infty}^\infty t^4\hat{g}''(t) e\left(\frac{\pi Y t}{T}\cos\left(\frac{\pi\alpha}{T}\right)\right) dt \nonumber\\&\hspace{1in}+ O\left(\frac{Y}{T^2} + \frac{Y^2}{T^3} + \frac{Y^3}{T^4}\right)
\\\label{actual seven}&= c_{10}T\sum_{|\alpha| < \frac{T}{2}} e\left(Y\sin\left(\frac{\pi\alpha}{T}\right)\right) \tilde{\tilde{g}}\left(\frac{\pi Y}{T}\cos\left(\frac{\pi\alpha}{T}\right)\right) \nonumber\\&\hspace{.5in}+ c_{11}\frac{Y}{T}\sum_{|\alpha| < \frac{T}{2}} e\left(Y\sin\left(\frac{\pi\alpha}{T}\right)\right) \sin\left(\frac{\pi\alpha}{T}\right) \tilde{\tilde{g}}''\left(\frac{\pi Y}{T}\cos\left(\frac{\pi\alpha}{T}\right)\right) \nonumber\\&\hspace{.5in}+ c_{12}\frac{Y}{T^2}\sum_{|\alpha| < \frac{T}{2}} e\left(Y\sin\left(\frac{\pi\alpha}{T}\right)\right) \cos\left(\frac{\pi\alpha}{T}\right) \tilde{\tilde{g}}'''\left(\frac{\pi Y}{T}\cos\left(\frac{\pi\alpha}{T}\right)\right) \nonumber\\&\hspace{.5in}+ c_{13}\frac{Y^2}{T^3}\sum_{|\alpha| < \frac{T}{2}} e\left(Y\sin\left(\frac{\pi\alpha}{T}\right)\right) \sin^2\left(\frac{\pi\alpha}{T}\right) \tilde{\tilde{g}}''''\left(\frac{\pi Y}{T}\cos\left(\frac{\pi\alpha}{T}\right)\right) \nonumber\\&\hspace{1in}+ O\left(\frac{Y}{T^2}\right).
\end{align}

As the rest of the argument is a bit long, we isolate it in Lemma \ref{stationary phase} immediately below. Its proof uses Poisson summation again. By \eqref{actual seven}, this finishes the proof of Proposition \ref{prop:smallboundSJ} (and hence that of Proposition \ref{small bound} as well). \end{proof}

\begin{lem}\label{stationary phase} Let $g$ be as in \eqref{definition of g}, and $Y\leq \frac{T}{2\pi}$. Then
\begin{eqnarray}\label{one sum}
& & (1) \ \ A_g^{(1)}(Y)\ :=\ T\sum_{|\alpha| < \frac{T}{2}} e\left(Y\sin\left(\frac{\pi\alpha}{T}\right)\right) \tilde{\tilde{g}}\left(\frac{\pi Y}{T}\cos\left(\frac{\pi\alpha}{T}\right)\right)\ \ll\  \frac{Y^4}{T^7} \nonumber\\ & & (2) \ \ A_g^{(2)}(Y)\ :=\ \frac{Y}{T}\sum_{|\alpha| < \frac{T}{2}} e\left(Y\sin\left(\frac{\pi\alpha}{T}\right)\right) \sin\left(\frac{\pi\alpha}{T}\right) \tilde{\tilde{g}}''\left(\frac{\pi Y}{T}\cos\left(\frac{\pi\alpha}{T}\right)\right)\ \ll\  \frac{Y^5}{T^9}.\nonumber\\ & & (3) \ \ A_g^{(3)}(Y)\ :=\ \frac{Y}{T^2}\sum_{|\alpha| < \frac{T}{2}} e\left(Y\sin\left(\frac{\pi\alpha}{T}\right)\right) \cos\left(\frac{\pi\alpha}{T}\right) \tilde{\tilde{g}}'''\left(\frac{\pi Y}{T}\cos\left(\frac{\pi\alpha}{T}\right)\right)\ \ll\  \frac{Y^5}{T^{10}}.\nonumber\\ & & (4) \ \ A_g^{(4)}(Y)\ :=\ \frac{Y^2}{T^3}\sum_{|\alpha| < \frac{T}{2}} e\left(Y\sin\left(\frac{\pi\alpha}{T}\right)\right) \sin^2\left(\frac{\pi\alpha}{T}\right) \tilde{\tilde{g}}''''\left(\frac{\pi Y}{T}\cos\left(\frac{\pi\alpha}{T}\right)\right)\ \ll\  \frac{Y^6}{T^{11}}.
\end{eqnarray}\end{lem}

\begin{proof}[Proof of Lemma \ref{stationary phase}]
We present the calculation for $A_g^{(1)}(Y)$ --- the same calculations work for $A_g^{(2)}(Y), A_g^{(3)}(Y)$, and $A_g^{(4)}(Y)$ upon inserting a $\sin$, $\cos$, or $\sin^2$ into the sum and replacing $\tilde{\tilde{g}}$ with one of its derivatives. Let $p\in C^\infty\left(\left[-\frac{T}{2},\frac{T}{2}\right]\right)$ such that $p\vert_{\left[-\frac{T-1}{2},\frac{T-1}{2}\right]} = 1$. We view $p$ as a Schwartz function on $\R$. Then
\begin{align}\label{one sum two}
A_g^{(1)}(Y)\ =\ T\sum_{\alpha\in \Z} p(\alpha) e\left(Y\sin\left(\frac{\pi\alpha}{T}\right)\right) \tilde{\tilde{g}}\left(\frac{\pi Y}{T}\cos\left(\frac{\pi\alpha}{T}\right)\right).
\end{align}
Applying Poisson summation,
\begin{eqnarray}\label{one sum three}
A_g^{(1)}(Y) &\ =\ & T\sum_{n\in \Z} \int_{-\frac{T}{2}}^{\frac{T}{2}} p(t) \tilde{\tilde{g}}\left(\frac{\pi Y}{T}\cos\left(\frac{\pi t}{T}\right)\right) e\left(Y\sin\left(\frac{\pi t}{T}\right)-nt\right) dt \nonumber\\ &=:& T\sum_{n\in \Z} B_g(Y,n).
\end{eqnarray}
For each $n$, the derivative of the phase in $B_g(Y,n)$ is \begin{equation}\frac{\pi Y}{T}\cos\left(\frac{\pi t}{T}\right) - n.\end{equation} Here is where our hypothesis on $Y$ (n\'{e}e $X$) comes in: for $Y\leq \frac{T}{2\pi}$ and $n\neq 0$, we have \begin{equation}\left|\frac{\pi Y}{T}\cos\left(\frac{\pi t}{T}\right) - n\right|\ \gg\ n.\end{equation}

Now we integrate by parts four times. There is nothing special about four other than the fact that the first four derivatives of $\tilde{\tilde{g}}$ have far more than four zeros at $0$ and $\sum n^{-4}$ converges. Integrating by parts more times would give us no improvement in the end. First consider the $n=0$ term of \eqref{one sum three} --- i.e., $B_g(Y,0)$ --- where the phase \emph{is} stationary (albeit at a boundary point of the integration region).
\begin{eqnarray}
B_g(Y,0) &\ =\ & \int_{-\frac{T}{2}}^{\frac{T}{2}} p(t)\tilde{\tilde{g}}\left(\frac{\pi Y}{T}\cos\left(\frac{\pi t}{T}\right)\right)e\left(Y\sin\left(\frac{\pi t}{T}\right)\right)dt\nonumber\\&=& -\int_{-\frac{T}{2}}^{\frac{T}{2}} \left(p(t)\tilde{g}\left(\frac{\pi Y}{T}\cos\left(\frac{\pi t}{T}\right)\right)\right)' e\left(Y\sin\left(\frac{\pi t}{T}\right)\right)dt.
\end{eqnarray}
Note that the $g$ has lost one tilde because we have divided out by the derivative of the phase, and also that the boundary terms vanish thanks to the support condition on $p$.

We remark before we repeat this three more times that $p' = 0$ on $\left[-\frac{T-1}{2},\frac{T-1}{2}\right]$, and on $\pm \left[\frac{T-1}{2},\frac{T}{2}\right]$ we have that \begin{equation}\tilde{\tilde{g}}\left(\frac{\pi Y}{T}\cos\left(\frac{\pi t}{T}\right)\right)\ \ll\ \left(\frac{Y}{T^2}\right)^{8},\end{equation} for instance (since, again, $g$ has a high order zero at $0$). Thus the terms with derivatives on $p$ are negligible. Further, differentiating the $\tilde{\tilde{g}}$ term picks up a factor of $Y/T^2$ (the same goes for any $\sin$, $\cos$, or $\sin^2$ terms as well), and differentiating the denominator we absorbed earlier would again pick up a factor of $Y/T^2$. The point is that, no matter which we differentiate, repeating this process three more times gives us a bound of the form \begin{equation}B_g(Y,0)\ \ll\ \left(\frac{Y}{T^2}\right)^4.\end{equation}

The exact same argument works for $n\neq 0$, except now we pick up at least one factor of $n$ each time we integrate by parts (since the derivative of the phase is $\frac{\pi Y}{T}\cos\left(\frac{\pi t}{T}\right) - n$). The same process and reasoning leads us to a bound of shape: \begin{equation}A_g(Y)\ \ll\ T\left(\frac{Y}{T^2}\right)^4\left(1 + \sum_{n\neq 0} \frac{1}{n^4}\right)\ \ll\ \frac{Y^4}{T^7},\end{equation} as desired.
\end{proof}

\subsection{Handling the remaining large primes}

The goal of this subsection is to prove Proposition \ref{large bound}. For this we apply the following asymptotic expansion, due to Dunster \cite{Du} and (essentially) found in Sarnak-Tsimerman \cite{ST}.

\begin{lem}
Let $x,r > 0$. Then
\begin{align}\label{dunster asymptotic}
J_{2ir}(x)\ =\ \frac{c_{14}e^{2ir\xi\left(\frac{x}{2r}\right)}}{\left(4r^2 + x^2\right)^{\frac{1}{4}}} e^{\pi r}\left(1 + \frac{1}{8\sqrt{4r^2+x^2}} - \frac{5r^2}{6\left(4r^2 + x^2\right)^{\frac{3}{2}}}\right) + O\left(\frac{e^{\pi r}}{\left(4r^2 + x^2\right)^{\frac{5}{4}}} + \frac{e^{-\pi r}}{\left(4r^2 + x^2\right)^{\frac{1}{4}}}\right),
\end{align}
where $\xi(z):=(1+z^2)^{\frac{1}{2}} + \log\left(\frac{z}{1+\sqrt{1+z^2}}\right)$.
\end{lem}

\begin{proof}[Proof of Proposition \ref{large bound}]
Write \begin{equation}D_J(X)\ :=\ \intii J_{2ir}(X) \frac{r h_T(r)}{\cosh\left(\pi r\right)} dr\end{equation} for our integral.

Observe that \begin{equation}\left(rT\xi\left(\frac{X}{2rT}\right)\right)' \ =\ T\log\left(-\frac{2rT}{X} + \sqrt{1 + \frac{4r^2T^2}{X^2}}\right),\end{equation} and that \begin{equation}\left(\frac{1}{\left(rT\xi\left(\frac{X}{2rT}\right)\right)'}\right)' \ =\ \frac{2}{T\sqrt{4r^2 + \frac{X^2}{T^2}}\left(\log\left(-\frac{2rT}{X} + \sqrt{1 + \frac{4r^2T^2}{X^2}}\right)\right)^2}.\end{equation} We will also use the fact that \begin{equation}T\log\left(-\frac{2rT}{X} + \sqrt{1 + \frac{4r^2T^2}{X^2}}\right) \ \gg\ \min\left(T, \frac{rT^2}{X}\right).\end{equation} Applying the asymptotic expansion of \eqref{dunster asymptotic} (and using evenness after splitting into positive and negative $r$), we see that $D_J(X)\ll D_J^+(X)$, with
\begin{align}\label{large sum two}
D_J^+(X)&\ :=\ c_{15}\int_{\R^+} \frac{e^{2ir\xi\left(\frac{X}{2r}\right)}r\tilde{h}\left(\frac{ir}{T}\right)}{\left(4r^2 + X^2\right)^{\frac{1}{4}} \sinh\left(\frac{\pi r}{T}\right)} dr \nonumber\\ &\hspace{.5in}+ c_{16}\int_{\R^+} \frac{e^{2ir\xi\left(\frac{X}{2r}\right)}r\tilde{h}\left(\frac{ir}{T}\right)}{\left(4r^2 + X^2\right)^{\frac{3}{4}} \sinh\left(\frac{\pi r}{T}\right)} \\&\hspace{.5in}+ c_{17}\int_{\R^+} \frac{e^{2ir\xi\left(\frac{X}{2r}\right)}r^3\tilde{h}\left(\frac{ir}{T}\right)}{\left(4r^2 + X^2\right)^{\frac{7}{4}} \sinh\left(\frac{\pi r}{T}\right)} \nonumber\\ &\hspace{1in}+ O\left(\int_{\R^+} \frac{\left|r\tilde{h}\left(\frac{ir}{T}\right)\right|}{\left(4r^2 + X^2\right)^{\frac{5}{4}}\sinh\left(\frac{\pi r}{T}\right)} dr\right)\nonumber\\&\ =:\ N_J^{(1)}(X) + N_J^{(2)}(X) + N_J^{(3)}(X) + E_J(X),
\end{align}
where the spacing is to indicate orders of growth.

Using our hypothesis on $X$ (and the exponential decay of $h_T$ at $\infty$), \begin{equation}E_J(X)\ \ll\ \frac{T^2}{X^{5/2}}.\end{equation} (To see this split the integral into $r\leq \frac{X}{T}$ and $r > \frac{X}{T}$.) Thus it suffices to study the first three terms of \eqref{large sum two} --- i.e., $N_J^{(i)}(X)$. We will work with $N_J^{(1)}(X)$, but the other two follow in exactly the same manner. Via $r\mapsto Tr$ and then integrating by parts $K\leq M$ times, we see that
\begin{eqnarray}
N_J^{(1)}(X)&\ = \ & c_{18}T^{\frac{3}{2}}\int_{\R^+} e\left(\frac{rT}{\pi}\xi\left(\frac{X}{2rT}\right)\right) \left(\frac{r\tilde{h}(ir)/\sinh(\pi r)}{\left(4r^2 + \frac{X^2}{T^2}\right)^{\frac{1}{4}}\left(rT\xi\left(\frac{X}{2rT}\right)\right)'}\right)' dr
\nonumber\\&= & c_{19}T^{\frac{3}{2}}\int_{\R^+} e\left(\frac{rT}{\pi}\xi\left(\frac{X}{2rT}\right)\right) \left(\frac{\left(\frac{r\tilde{h}(ir)/\sinh(\pi r)}{\left(4r^2 + \frac{X^2}{T^2}\right)^{\frac{1}{4}}\left(rT\xi\left(\frac{X}{2rT}\right)\right)'}\right)'}{\left(rT\xi\left(\frac{X}{2rT}\right)\right)'}\right)' dr
\nonumber\\&\ll &\frac{X^{K-\frac{1}{2}}}{T^{2K-2}}.
\end{eqnarray}
Note that integrating by parts twice is sufficient to break $(-1,1)$. In any case, let us explain the final bound above. In the numerator we start off with $T^{\frac{3}{2}}$ on the outside. After integrating by parts once (the first line), the worst case occurs when we differentiate the $r\tilde{h}(ir)/\sinh(\pi r)$ term (else we gain powers of $X/T$ in the denominator in the final bound). In this case, let us consider the denominator. When $r\ll X/T$ we have an $X^\frac{1}{2} / T^{\frac{1}{2}}$ from the first term, and an $rT^2/X$ from the second. When $r$ is large we have an $r^{\frac{1}{2}}$ from the first term and a $T$ from the second. The numerator decays exponentially and absorbs the $r$ in the denominator when $r$ is small (and this is the only constraint on repeating the integration by parts), since $r\tilde{h}(ir)$ has $M+2$ zeros at $0$. Therefore the bound has moved from the trivial bound of $X^{-\frac{1}{2}}T$ to $\ll T^{\frac{3}{2} + \frac{1}{2} - 2} X^{1 - \frac{1}{2}} + 1 = X^{\frac{1}{2}}$. And indeed this pattern continues --- the dominant part of the integral is that with $r\leq \frac{X}{T}$, due to the exponential decay of $r\tilde{h}(ir)/\sinh(\pi r)$. In this regime integration by parts picks up a factor of $rT^2/X$ in the denominator, with the $r$ absorbed into $r\tilde{h}(ir)$, thus gaining $X/T^2$ in total. We may repeat this as many times as $r\tilde{h}(ir)$ has zeros divided by 2 (since we are also differentiating), which is $M+1$ times. We will choose $K = M$ in any case.

Note that, applying the same procedure, $N_J^{(2)}(X)$ and $N_J^{(3)}(X)$ contribute to lower order (namely, we gain at least factor of $T$ in each case). Therefore, taking $K = M$, our final bound is
\begin{eqnarray}\label{large sum two}
D_J(X)& \ \ll\ & \frac{X^{M - \frac{1}{2}}}{T^{2M-2}} + \frac{T^2}{X^{5/2}}.
\end{eqnarray}
This completes the proof.
\end{proof}

\section{Proof of Theorem \ref{main theorem}}

We can now prove our main result.

\begin{proof}[Proof of Theorem \ref{main theorem}]
We prove part (2) of Lemma \ref{conductor, prime sum, and prime square sum}. Part (3) follows entirely analogously (in fact, we obtain better bounds in this case).

We have already seen that the total mass of the averages is on the order of $T^2$. So it suffices to give a bound of size $o(T^2)$ for
\begin{align}\label{kill this}
\sum_p \frac{\log{p}}{p^{1/2}\log{T}}\hat{\phi}\left(\frac{\log{p}}{2\log{T}}\right)\sum_{u\in \mathcal{M}_1} \frac{h_T(t_u)}{||u||^2} \lambda_p(u).
\end{align}
Applying the Kuznetsov trace formula and using the same arguments used for \eqref{Eisenstein term} gives us that
\begin{align}\label{kill this two}
\sum_{u\in \mathcal{M}_1} \frac{h_T(t_u)}{||u||^2} \lambda_p(u) = c_{20}\sum_{c\geq 1} \frac{S(1,p;c)}{c}\intii J_{2ir}\left(\frac{4\pi\sqrt{p}}{c}\right) \frac{r h_T(r)}{\cosh(\pi r)} dr + O\left(\frac{T\log{T}}{p^{1/2}}\right).
\end{align}
Since $\phi$ has compact support, the sum of the error term over the primes is
\begin{align}
\sum_p \frac{\log{p}}{p^{1/2}\log{T}}\hat{\phi}\left(\frac{\log{p}}{2\log{T}}\right)O\left(\frac{T\log{T}}{p^{1/2}}\right)\ \ll\ T\log{T}.
\end{align}
We split the remaining double sum into three parts as follows.
\begin{align}
&\sum_p \frac{\log{p}}{p^{1/2}\log{T}}\hat{\phi}\left(\frac{\log{p}}{2\log{T}}\right)\sum_{c\geq 1} \frac{S(1,p;c)}{c}\intii J_{2ir}\left(\frac{4\pi\sqrt{p}}{c}\right) \frac{r h_T(r)}{\cosh(\pi r)} dr
\nonumber\\&= \label{kill this three}\sum_{\frac{T^2}{4\pi^2}\leq p\leq T^{2\eta}} \frac{\log{p}}{p^{1/2}\log{T}}\hat{\phi}\left(\frac{\log{p}}{2\log{T}}\right)\sum_{c\leq \frac{4\pi\sqrt{p}}{T}} \frac{S(1,p;c)}{c}\intii J_{2ir}\left(\frac{4\pi\sqrt{p}}{c}\right) \frac{r h_T(r)}{\cosh(\pi r)} dr
\nonumber\\&\quad+ \sum_{\frac{T^2}{4\pi^2}\leq p\leq T^{2\eta}} \frac{\log{p}}{p^{1/2}\log{T}}\hat{\phi}\left(\frac{\log{p}}{2\log{T}}\right)\sum_{c > \frac{4\pi\sqrt{p}}{T}} \frac{S(1,p;c)}{c}\intii J_{2ir}\left(\frac{4\pi\sqrt{p}}{c}\right) \frac{r h_T(r)}{\cosh(\pi r)} dr
\nonumber\\&\quad+ \sum_{p < \frac{T^2}{4\pi^2}} \frac{\log{p}}{p^{1/2}\log{T}}\hat{\phi}\left(\frac{\log{p}}{2\log{T}}\right)\sum_{c\geq 1} \frac{S(1,p;c)}{c}\intii J_{2ir}\left(\frac{4\pi\sqrt{p}}{c}\right) \frac{r h_T(r)}{\cosh(\pi r)} dr.
\end{align}

We apply the Weil bound for Kloosterman sums to each: $|S(1,p;c)| \ll c^{1/2+\epsilon}$. Moreover, we apply Proposition \ref{large bound} to the integrals in the first sum of \eqref{kill this three}, and Proposition \ref{small bound} to those in the second and third sums of \eqref{kill this three}. We get that \eqref{kill this three} is bounded by
\begin{align}\label{kill this four}
&T^{-2M+2}\sum_{\frac{T^2}{4\pi^2}\leq p\leq T^{2\eta}} \frac{p^{\frac{M}{2}-\frac{3}{4}}\log{p}}{\log{T}}\hat{\phi}\left(\frac{\log{p}}{2\log{T}}\right)\sum_{c\leq \frac{4\pi\sqrt{p}}{T}} c^{-M+\eps}
\nonumber\\+\hspace{.2cm}& T^2\sum_{\frac{T^2}{4\pi^2}\leq p\leq T^{2\eta}} \frac{\log{p}}{p^{\frac{7}{4}}\log{T}}\hat{\phi}\left(\frac{\log{p}}{2\log{T}}\right)\sum_{c\leq \frac{4\pi\sqrt{p}}{T}} c^{2+\eps}
\nonumber\\+\hspace{.2cm}& T^{-2}\sum_{\frac{T^2}{4\pi^2}\leq p\leq T^{2\eta}} \frac{\log{p}}{\log{T}}\hat{\phi}\left(\frac{\log{p}}{2\log{T}}\right)\sum_{c > \frac{4\pi\sqrt{p}}{T}} c^{-\frac{3}{2} + \eps}
\nonumber\\+\hspace{.2cm}& T^{-2}\sum_{p < \frac{T^2}{4\pi^2}} \frac{\log{p}}{\log{T}}\hat{\phi}\left(\frac{\log{p}}{2\log{T}}\right)\sum_{c\geq 1} c^{-\frac{3}{2} + \eps}.
\end{align}
Applying Chebyshev's prime number theorem estimates, \eqref{kill this four} is
\begin{align}
&\ll\ \frac{T^{\left(M + \frac{1}{2}\right)\eta - 2M + 2 + \eps}}{\log{T}} + \frac{T^{\frac{3}{2}\eta - 1 + \eps}}{\log{T}} + \frac{T^{\frac{3}{2}\eta-\frac{3}{2} + \eps}}{\log{T}} + \frac{T^\eps}{\log{T}},
\end{align}
which is of the desired shape (that is, $o(T^2)$) when $\eta < 2 - \frac{1}{M + \frac{1}{2}}$, completing the argument.
\end{proof}

\section*{Acknowledgements}

The first-named author was partially supported by NSF grant DMS0850577 and the second-named author by NSF grants DMS0970067 and DMS1265673. It is a pleasure to thank Andrew Knightly, Peter Sarnak, and our colleagues from the Williams College 2011 and 2012 SMALL REU programs for many helpful conversations.

\ \\

\end{document}